\documentclass[10pt]{article}
\usepackage[T1]{fontenc}
\usepackage{lmodern}
\usepackage{amsmath,amssymb,amsthm}
\usepackage[a4paper,margin=25mm]{geometry}
\usepackage{microtype,booktabs,enumitem}
\usepackage{xcolor}
\usepackage{hyperref}
\hypersetup{
  colorlinks=true,
  allcolors=blue!50!black,
  pdftitle={Degreewise Cut-Semigroup Saturation for $K_5$-Minor-Free Graphs and Seymour's Planar Edge-Colouring Conjecture}
}
\newtheorem{theorem}{Theorem}[section]
\newtheorem{lemma}[theorem]{Lemma}
\newtheorem{proposition}[theorem]{Proposition}
\newtheorem{corollary}[theorem]{Corollary}
\theoremstyle{definition}
\newtheorem{definition}[theorem]{Definition}
\theoremstyle{remark}
\newtheorem{remark}[theorem]{Remark}
\DeclareMathOperator{\Cut}{Cut}
\DeclareMathOperator{\Sat}{Sat}
\DeclareMathOperator{\Dec}{Dec}
\DeclareMathOperator{\conv}{conv}
\DeclareMathOperator{\PM}{PM}
\DeclareMathOperator{\cone}{cone}
\newcommand{\ZZ}{\mathbb Z}
\newcommand{\RR}{\mathbb R}
\newcommand{\NN}{\mathbb Z_{\geq0}}
\newcommand{\EC}{\mathcal E}
\newcommand{\symd}{\mathbin{\triangle}}
\title{Degreewise Cut-Semigroup Saturation for $K_5$-Minor-Free Graphs and Seymour's Planar Edge-Colouring Conjecture}
\author{SeungJu Lee}
\date{\today}
\begin{document}
\maketitle
\begin{abstract}
For every positive integer $k$, we prove that the homogeneous cut semigroup of every $K_5$-minor-free graph is saturated at height $k$ if and only if every planar $k$-graph is $k$-edge-colourable.
Here a $k$-graph is a loopless $k$-regular multigraph in which every odd vertex cut has size at least $k$.
A triangle expansion of a cubic plane dual converts cut decompositions into perfect-matching decompositions.
The converse uses symmetric difference with a fixed perfect matching.
The equivalence identifies the normality conjecture for $K_5$-minor-free cut polytopes with Seymour's planar edge-colouring conjecture.
In particular, the known cases $k\leq8$ give saturation through height eight.
\end{abstract}
\noindent\textbf{Keywords:} cut polytope; affine semigroup; planar graph; perfect matching; edge-colouring.\par
\smallskip
\noindent\textbf{2020 Mathematics Subject Classification:} 05C15, 05C70, 52B20, 20M25.

\section{Introduction}
All graphs are finite.
Unless specified otherwise, graphs are simple; multigraphs may have parallel edges.
A plane graph is a graph with a fixed embedding in the sphere.
For a graph $G=(V,E)$ and $S\subseteq V$, let $\delta_G(S)$ denote its cut edge set, and let $\chi_{\delta_G(S)}\in\{0,1\}^{E}$ be its incidence vector.
Write
\[
 \Cut(G)=\conv\{\chi_{\delta_G(S)}:S\subseteq V\},\qquad
 A_G=\{(\chi_{\delta_G(S)},1):S\subseteq V\}.
\]
The homogeneous cut semigroup is $\mathsf S_G=\NN A_G$, with lattice $\ZZ A_G$.
For $k\geq1$, put
\begin{align*}
 \Sat_k(G)&=\{x\in\ZZ^E:(x,k)\in\ZZ A_G\cap\cone(A_G)\},\\
 \Dec_k(G)&=\left\{\sum_{i=1}^k\chi_{\delta_G(S_i)}:S_i\subseteq V\right\}.
\end{align*}
We call $G$ \emph{saturated at height $k$} if $\Sat_k(G)=\Dec_k(G)$.
Repetitions and the empty cut are allowed.
The semigroup is normal precisely when this equality holds at every positive height: its height-zero cone slice consists only of the origin.

Sturmfels and Sullivant conjectured that $\mathsf S_G$ is normal if and only if $G$ has no $K_5$ minor \cite[Conjecture~3.7]{SS}.
The implication from normality to minor exclusion is known.
Ohsugi proved that normality is minor closed \cite[Corollary~2.4]{Ohsugi} and reduced the converse to 4-connected plane triangulations \cite[Section~4]{Ohsugi}.

\begin{definition}
A \emph{$k$-graph} is a loopless $k$-regular multigraph $R$ such that
\begin{equation}\label{eq:oddcut}
 |\delta_R(X)|\geq k\quad\text{for every }X\subseteq V(R)\text{ of odd cardinality}.
\end{equation}
Degrees and cut sizes count multiplicities.
The condition applies to all odd vertex sets, including $V(R)$, and implies that every component has even order.
\end{definition}

Let $\mathsf{Sey}_k$ be the assertion that every planar $k$-graph has a proper edge-colouring with $k$ colours.
This is the degree-$k$ form of Seymour's conjecture as stated in \cite[Conjecture~1.1]{CES}.
On a $k$-regular graph such a colouring is a partition of its edges into $k$ perfect matchings.

\begin{theorem}\label{thm:main}
Fix $k\geq1$.
The following statements are equivalent.
\begin{enumerate}[label=\textup{(\roman*)}]
\item Every simple $K_5$-minor-free graph is saturated at height $k$.
\item Every simple planar graph is saturated at height $k$.
\item Every planar $k$-graph is $k$-edge-colourable.
\end{enumerate}
\end{theorem}

Laso\'n and Micha\l{}ek proved the equivalence between height-three saturation for planar graphs and the four-colour theorem \cite[Theorem~4]{LM}.
They also established planar saturation at heights one and two \cite[Corollary~5]{LM}.
To treat arbitrary $k$, we first prove the result for plane triangulations and then apply Ohsugi's reduction, preserving the height at each step.

\section{Lattices, duality and perfect matchings}\label{sec:prelim}
For an edge vector $x$ and an edge set $F$, write $x(F)=\sum_{e\in F}x_e$.
An edge set is \emph{Eulerian} if its spanning subgraph has even degree at every vertex.
Write $\EC(H)$ for the set of incidence vectors of Eulerian edge sets of $H$.

\begin{lemma}[{\cite[equation~(1)]{Ohsugi}}]\label{lem:lattice}
For a simple graph $G$, $x\in\ZZ^{E(G)}$ and $k\in\ZZ$, the following are equivalent:
\begin{enumerate}[label=\textup{(\roman*)}]
\item $(x,k)\in\ZZ A_G$;
\item $x(C)$ is even for every cycle $C$ of $G$;
\item $x\bmod2$ is a cut vector over $\mathbb F_2$.
\end{enumerate}
For $k\geq1$, also
\begin{equation}\label{eq:coneslice}
 (x,k)\in\cone(A_G)\quad\Longleftrightarrow\quad x/k\in\Cut(G).
\end{equation}
\end{lemma}
\begin{proof}
A cut meets each cycle in an even number of edges, proving (i)$\Rightarrow$(ii).
To prove (ii)$\Rightarrow$(iii), choose a spanning forest and label a root in each component by zero.
Label any other vertex by the sum modulo two of the edge coordinates along its forest path from the root.
The fundamental-cycle conditions show that the labels at the endpoints of every edge differ by $x_e\bmod2$.
The vertices labelled one define the required cut.

For (iii)$\Rightarrow$(i), let $d$ be this cut vector, so $x-d\in2\ZZ^{E(G)}$.
Writing $\mathbf e_e$ for the coordinate vector of an edge $e=uv$, we have
\[
 2\mathbf e_e=\chi_{\delta_G(\{u\})}
 +\chi_{\delta_G(\{v\})}-\chi_{\delta_G(\{u,v\})}.
\]
The corresponding combination of lifted cuts has height one; subtracting the empty lifted cut makes its height zero.
Hence $(2\mathbf e_e,0)\in\ZZ A_G$.
Writing $x-d=2z$, we obtain
\[
 (x,k)=(d,1)+\sum_{e\in E(G)}z_e(2\mathbf e_e,0)+(k-1)(0,1)\in\ZZ A_G.
\]
Equation~\eqref{eq:coneslice} follows by dividing the coefficients of a conic combination by their sum $k$.
\end{proof}

\begin{lemma}[Plane cut--cycle duality]\label{lem:duality}
Let $H$ be a connected loopless plane multigraph and let $G_0=H^*$ be its dual, possibly with loops and parallel edges.
Under the edge correspondence, cut edge sets of $G_0$ are exactly Eulerian edge sets of $H$.
\end{lemma}
\begin{proof}
Over $\mathbb F_2$, the boundary of each face of $H$ is an even edge set, with repeated traversals cancelled.
It is the cut of the corresponding dual vertex.
These boundaries span the cycle space: the dual is connected, so their span has dimension $|F(H)|-1$, and Euler's formula identifies this number with $|E(H)|-|V(H)|+1$, the dimension of the cycle space of $H$.
Dual cuts are sums modulo two of dual vertex cuts.
They therefore correspond to precisely the Eulerian edge sets of $H$.
In particular, bridges of $H$ correspond to dual loops and belong to none of these edge sets.
\end{proof}

For a simple graph $H$, let $\PM(H)$ be the convex hull of its perfect-matching incidence vectors.
Edmonds's theorem gives
\begin{equation}\label{eq:edmonds}
 \PM(H)=\{y\in\RR^{E(H)}:y\geq0,\ y(\delta_H(v))=1\ (v\in V(H)),\
 y(\delta_H(X))\geq1\ (|X|\text{ odd})\}.
\end{equation}
Indeed, Edmonds describes the matching polytope by nonnegativity, vertex sums at most one, and $y(E_H(X))\leq(|X|-1)/2$ for odd $X$ \cite[Section~2, Theorem~(P), pp.~125--126]{Edmonds}.
On the face where every vertex sum is one,
\[
 2y(E_H(X))+y(\delta_H(X))=|X|.
\]
The odd-set inequalities are therefore equivalent to those in \eqref{eq:edmonds}.
Since every matching in a convex combination on this face is perfect, the face is $\PM(H)$.

Thus a loopless multigraph with underlying simple graph $H$ and multiplicity vector $m$ is a $k$-graph if and only if $m/k\in\PM(H)$.

\section{Triangle expansion}\label{sec:gadget}
Let $H$ be a loopless cubic multigraph.
Replace each vertex $v$, with incident edges $a,b,c$, by a triangle with vertices $(v,a),(v,b),(v,c)$.
For each original edge $e=uv$, add the \emph{external} edge $(u,e)(v,e)$.
The triangle edges are \emph{internal}.
Denote the resulting graph by $\Phi(H)$.
If $H$ is plane, use the cyclic order of its incidences to obtain a plane embedding of $\Phi(H)$.

\begin{lemma}\label{lem:bijection}
There is a bijection $C\mapsto P(C)$ between Eulerian edge sets of $H$ and perfect matchings of $\Phi(H)$.
\end{lemma}
\begin{proof}
An Eulerian edge set $C$ has degree zero or two at each vertex of $H$.
Include the external edge of $e$ precisely when $e\notin C$.
At a vertex where $C$ uses two edges, include the internal edge joining their two incidence vertices.
If $C$ uses no edge there, all three incidence vertices are matched externally.
This defines a perfect matching.

Conversely, in a perfect matching the number of externally matched vertices of each triangle is odd, hence one or three.
In the first case the other two vertices are joined by their internal edge; in the second there is no internal edge.
Declaring $e\in C$ exactly when its external edge is absent gives degree two or zero at every original vertex.
The two constructions are inverse.
\end{proof}

Fix $k\geq1$ and an integral edge vector $p$ satisfying
\begin{equation}\label{eq:admissible}
 p/k\in\conv\EC(H),\qquad p(\delta_H(v))\equiv0\pmod2\quad(v\in V(H)).
\end{equation}
At a vertex with incident edges $a,b,c$, define
\begin{align}
 \lambda_0(v)&=k-\tfrac12(p_a+p_b+p_c),\label{eq:lambda0}\\
 \lambda_a(v)&=\tfrac12(-p_a+p_b+p_c),\quad
 \lambda_b(v)=\tfrac12(p_a-p_b+p_c),\quad
 \lambda_c(v)=\tfrac12(p_a+p_b-p_c).\label{eq:lambda}
\end{align}
The parity condition makes these numbers integral.
To see that they are nonnegative, choose an Eulerian set at random with expected incidence vector $p/k$.
At $v$, the possible edge sets are $\varnothing$, $\{b,c\}$, $\{a,c\}$ and $\{a,b\}$.
Solving for their probabilities from the three marginals and total probability one gives respectively $\lambda_0/k$, $\lambda_a/k$, $\lambda_b/k$ and $\lambda_c/k$.

Define a multigraph $R(H,p,k)$ on the vertices of $\Phi(H)$ by giving external edge $e$ multiplicity $k-p_e$, and the internal edge $(v,b)(v,c)$ multiplicity $\lambda_a(v)$, cyclically.
Omit edges of multiplicity zero.

\begin{proposition}\label{prop:gadget}
If $H$ is plane, loopless and cubic and \eqref{eq:admissible} holds, then $R(H,p,k)$ is a planar $k$-graph.
Moreover,
\begin{equation}\label{eq:local-equivalence}
 p=\sum_{i=1}^k\chi_{C_i}\text{ with each }C_i\text{ Eulerian}
 \quad\Longleftrightarrow\quad R(H,p,k)\text{ is }k\text{-edge-colourable}.
\end{equation}
\end{proposition}
\begin{proof}
All multiplicities are nonnegative integers, since also $0\leq p_e\leq k$.
At $(v,a)$ the total degree is
\[
 (k-p_a)+\lambda_b(v)+\lambda_c(v)=k.
\]
Planarity follows from the construction of $\Phi(H)$.
Apply Lemma~\ref{lem:bijection} to a convex combination representing $p/k$.
The expected external-edge incidence is $1-p_e/k$, and the expected internal-edge incidence opposite $a$ is $\lambda_a(v)/k$.
The multiplicity vector of $R$, divided by $k$, is therefore a convex combination of perfect matchings of $\Phi(H)$.
Since each perfect matching crosses every odd vertex cut, the multiplicity of every such cut in $R$ is at least $k$.

If $p=\sum_i\chi_{C_i}$, the matchings $P(C_i)$ use each external edge type exactly $k-p_e$ times.
Their local pattern counts are uniquely determined by \eqref{eq:lambda0}--\eqref{eq:lambda}, so each internal edge type is also used with its specified multiplicity.
Assign distinct parallel copies to these occurrences.
This gives a $k$-edge-colouring of $R$.

Conversely, the colour classes of a $k$-edge-colouring of $R$ are perfect matchings.
Project these matchings to $P_1,\ldots,P_k$ in $\Phi(H)$, and let $C_i$ correspond to $P_i$ under Lemma~\ref{lem:bijection}.
Exactly $k-p_e$ matchings use the external edge of $e$, so exactly $p_e$ of the sets $C_i$ contain $e$.
This proves \eqref{eq:local-equivalence}.
\end{proof}

\begin{corollary}\label{cor:semigroup}
For any loopless cubic multigraph $H$, the correspondence in Lemma~\ref{lem:bijection} extends to an isomorphism of graded affine semigroups
\[
 \NN\{(\chi_C,1):C\text{ Eulerian in }H\}
 \ \cong\ 
 \NN\{(\chi_P,1):P\text{ a perfect matching of }\Phi(H)\}.
\]
The induced maps on the generated lattices, cones and semigroup algebras are also isomorphisms.
\end{corollary}
\begin{proof}
Let $T(p,k)=(m(p,k),k)$, with external coordinates $m_e=k-p_e$ and internal coordinates given by \eqref{eq:lambda}.
These expressions are linear in $(p,k)$.
Lemma~\ref{lem:bijection} gives $T(\chi_C,1)=(\chi_{P(C)},1)$.
The inverse on the image recovers $p_e=k-m_e$ from external coordinates.
Hence $T$ restricts to bijections on the semigroups, their integer spans and their real cones.
It preserves addition and height.
For any field $\mathbb F$, the monomial map $X^{(p,k)}\mapsto Y^{T(p,k)}$ gives the corresponding isomorphism of semigroup algebras over $\mathbb F$.
\end{proof}

If $H$ is connected and plane, duality identifies the first semigroup with the homogeneous cut semigroup of $H^*$, allowing dual loops and parallel edges.

\section{From edge-colouring to cut saturation}\label{sec:forward}
\begin{lemma}\label{lem:triangulation}
Assume $\mathsf{Sey}_k$.
Every simple plane triangulation on at least three vertices is saturated at height $k$.
\end{lemma}
\begin{proof}
Let $G$ be such a triangulation and $H=G^*$ its loopless cubic dual.
For $x\in\Sat_k(G)$, transfer coordinates to a vector $p$ on $H$.
Lemmas~\ref{lem:lattice} and \ref{lem:duality} give \eqref{eq:admissible}.
Proposition~\ref{prop:gadget} and $\mathsf{Sey}_k$ give a sum of $k$ Eulerian sets with total $p$.
Duality converts it to a sum of $k$ cuts with total $x$.
\end{proof}

Ohsugi's clique-sum and edge-deletion arguments \cite[Theorems~3.2 and~2.3]{Ohsugi} preserve saturation at a fixed height, as the next two lemmas show.

\begin{lemma}[Clique sums]\label{lem:clique}
Suppose $G_1$ and $G_2$ intersect in a complete graph on a vertex set $Q$ of size at most three, and $G=G_1\cup G_2$ with no additional cross edges.
If $G_1$ and $G_2$ are saturated at height $k$, so is $G$.
\end{lemma}
\begin{proof}
Restrict $x\in\Sat_k(G)$ to each $G_i$ and decompose each restriction into $k$ cuts.
If $Q=\varnothing$, pair the cuts arbitrarily and take unions of their chosen sides.
Suppose $Q\ne\varnothing$ and fix $q_0\in Q$.
Complement chosen sides so that $q_0$ lies outside each one.

For $|Q|=1$ all boundary patterns agree.
For $Q=\{q_0,q_1\}$, the number of cuts with $q_1$ on the opposite side from $q_0$ is the common coordinate $x_{q_0q_1}$.
For $Q=\{q_0,q_1,q_2\}$, put $a=x_{q_0q_1}$, $b=x_{q_0q_2}$ and $c=x_{q_1q_2}$.
If $n_{ij}$ counts chosen sides $S$ with $\chi_S(q_1)=i$ and $\chi_S(q_2)=j$, then
\begin{equation}\label{eq:boundary}
 \begin{aligned}
 n_{00}&=k-\frac{a+b+c}{2},& n_{10}&=\frac{a+c-b}{2},\\
 n_{01}&=\frac{b+c-a}{2},& n_{11}&=\frac{a+b-c}{2}.
 \end{aligned}
\end{equation}
Thus the boundary-pattern multiplicities agree in both decompositions.
Reorder one list so that paired cuts agree on $Q$, and take unions of their chosen sides.
Their $k$ cut vectors sum to $x$ on every edge of $G$.
\end{proof}

\begin{lemma}[Edge deletion]\label{lem:deletion}
Let $G$ be simple and $K_5$-minor-free, and $e_0\in E(G)$.
If $G$ is saturated at height $k\geq1$, then $G\setminus e_0$ is saturated at height $k$.
\end{lemma}
\begin{proof}
Set $G'=G\setminus e_0$ and take $x\in\Sat_k(G')$.
We seek an integer $\gamma$ extending $x$ to a vector $x'$ on $G$ with $x'_{e_0}=\gamma$ and $x'\in\Sat_k(G)$.

By the Barahona--Mahjoub description of the cut polytope \cite[Proposition~2.1 and Corollary~2.2]{Ohsugi}, $x'/k\in\Cut(G)$ precisely when $0\leq x'_e\leq k$ and
\begin{equation}\label{eq:cycle}
 x'(F)-x'(E(C)\setminus F)\leq k(|F|-1)
\end{equation}
for each induced cycle $C$ and odd subset $F\subseteq E(C)$.
Inequalities not involving $e_0$ hold because $x/k\in\Cut(G')$.
For those involving $e_0$, define
\begin{align*}
 \ell(C,F)&=x(F)-x\bigl(E(C)\setminus(F\cup\{e_0\})\bigr)-k(|F|-1)
 &&(e_0\notin F),\\
 u(C,F)&=-x(F\setminus\{e_0\})+x(E(C)\setminus F)+k(|F|-1)
 &&(e_0\in F).
\end{align*}
They impose $\gamma\geq\ell(C,F)$ and $\gamma\leq u(C,F)$, respectively.
Let $L$ be the maximum of zero and all lower bounds, and $U$ the minimum of $k$ and all upper bounds.
Every cut of $G'$ extends to $G$ using the same vertex partition.
Hence a convex representation of $x/k$ provides a real extension, proving $L\leq U$.
Both endpoints are integers.

If $e_0$ lies on no cycle, any integer in $[L,U]$ has the required lattice property.
Otherwise, for a cycle $C$ through $e_0$, put
\[
 \pi\equiv x(E(C)\setminus\{e_0\})\pmod2.
\]
This value is independent of $C$: the symmetric difference of two such cycles is an even subgraph of $G'$, decomposable into cycles, on each of which $x$ has even sum.
If $L<U$, two consecutive integers lie in the interval and one has parity $\pi$.
If $L=U$, at least one active endpoint bound must be a cycle bound, since the coordinate bounds $0$ and $k$ are distinct.
Every cycle bound has parity $\pi$, because $|F|-1$ is even and signs disappear modulo two.
The unique endpoint therefore also has the required parity.

Choose this $\gamma$.
The inequalities give $x'/k\in\Cut(G)$, and all cycle sums are even, including those through $e_0$.
Lemma~\ref{lem:lattice} gives $x'\in\Sat_k(G)$.
Decompose $x'$ into $k$ cuts of $G$ and restrict them to $G'$.
\end{proof}

\begin{proposition}\label{prop:forward}
For every $k\geq1$, $\mathsf{Sey}_k$ implies saturation at height $k$ for every simple $K_5$-minor-free graph.
\end{proposition}
\begin{proof}
Graphs on fewer than three vertices have either no edge or one edge and satisfy the conclusion directly.
Otherwise add edges, keeping the same vertex set and excluding $K_5$ minors, until reaching an edge-maximal graph $\widehat G$.

By \cite[Proposition~4.1]{Ohsugi}, $\widehat G$ is a clique sum of copies of $K_3$, $K_4$, the Wagner graph $V_8$, and 4-connected plane triangulations, with each intersection of order at most three.
The cut semigroups of $K_3$, $K_4$ and $V_8$ are normal \cite[Section~4, Example~4.2]{Ohsugi}.

Every piece is therefore saturated at height $k$, by Lemma~\ref{lem:triangulation} in the triangulation case.
Repeated application of Lemma~\ref{lem:clique} gives saturation for $\widehat G$.
Delete the added edges one at a time and apply Lemma~\ref{lem:deletion}; every intermediate graph remains $K_5$-minor-free.
\end{proof}

\section{From cut saturation to edge-colouring}\label{sec:reverse}
Let $R$ be a nonempty connected planar $k$-graph.
Choose a plane embedding of its underlying simple graph $H$, and let $m_e\geq1$ record its edge multiplicities.
By \eqref{eq:edmonds}, $m/k\in\PM(H)$.
Choose a perfect matching $M_0$ of $H$ and define
\begin{equation}\label{eq:switch}
 p_e=\begin{cases}k-m_e,&e\in M_0,\\m_e,&e\notin M_0.\end{cases}
\end{equation}

\begin{lemma}\label{lem:switch}
The vector $p$ is integral, $p/k\in\conv\EC(H)$, and $p(\delta_H(v))$ is even for every vertex $v$.
\end{lemma}
\begin{proof}
Write $m/k=\sum_j\alpha_j\chi_{M_j}$ as a convex combination of perfect matchings.
By \eqref{eq:switch},
\[
 p/k=\sum_j\alpha_j\chi_{M_j\symd M_0}.
\]
Each symmetric difference is Eulerian.
If $e_0$ is the edge of $M_0$ incident with $v$, regularity gives
\[
 p(\delta_H(v))=(k-m_{e_0})+\sum_{e\in\delta_H(v)\setminus\{e_0\}}m_e
 =2(k-m_{e_0}).
\]
\end{proof}

Let $G_0=H^*$, allowing dual loops and parallel edges.
Lemmas~\ref{lem:duality} and \ref{lem:switch} show that $p/k$ is a convex combination of cuts of $G_0$, and that $p\bmod2$ itself is a cut of $G_0$.
It follows that $p$ is zero on loops and constant on parallel classes.
Delete the loops of $G_0$ and retain one edge from each parallel class, giving a simple planar graph $G$ on the same vertices and an induced vector $\bar p$.

\begin{lemma}\label{lem:simplify}
The vector $\bar p$ belongs to $\Sat_k(G)$.
Every decomposition of $\bar p$ into $k$ cuts of $G$ lifts to a decomposition of $p$ into $k$ Eulerian edge sets of $H$.
\end{lemma}
\begin{proof}
The convex combination and the binary cut both descend to $G$.
Lemma~\ref{lem:lattice} gives $\bar p\in\Sat_k(G)$.
Conversely, each cut of $G$ lifts using the same vertex partition to a cut of $G_0$.
The lifted sum agrees with $p$ on parallel classes and is zero on loops.
Apply Lemma~\ref{lem:duality}.
\end{proof}

\begin{proposition}\label{prop:reverse}
If every simple planar graph is saturated at height $k$, then every planar $k$-graph is $k$-edge-colourable.
\end{proposition}
\begin{proof}
First consider connected nonempty $R$.
By the assumption and Lemma~\ref{lem:simplify}, write $p=\sum_{i=1}^k\chi_{C_i}$ with each $C_i$ Eulerian in $H$.
Define $Q_i=M_0\symd C_i$.
Each $Q_i$ has odd degree at every vertex.
For every edge, \eqref{eq:switch} yields
\begin{equation}\label{eq:aggregate}
 \sum_{i=1}^k\chi_{Q_i}(e)=m_e.
\end{equation}
For $e\in M_0$, the left side is $k-p_e$; otherwise it is $p_e$.
Consequently,
\begin{equation}\label{eq:budget}
 \sum_{i=1}^k d_{Q_i}(v)=m(\delta_H(v))=k\quad(v\in V(H)).
\end{equation}
The left side consists of $k$ positive odd integers.
All are therefore one, so every $Q_i$ is a perfect matching.
For each support edge, assign its parallel copies bijectively to the indices where it occurs in \eqref{eq:aggregate}.
The resulting $k$ perfect matchings partition $E(R)$.

Every component of a disconnected $k$-graph has even order: an odd component would violate \eqref{eq:oddcut}.
Each component inherits the same condition, so the connected construction applies componentwise with the same colour set.
The empty graph is immediate.
\end{proof}

\begin{proof}[Proof of Theorem~\ref{thm:main}]
Proposition~\ref{prop:forward} proves (iii)$\Rightarrow$(i), and Proposition~\ref{prop:reverse} proves (ii)$\Rightarrow$(iii).
Planarity gives (i)$\Rightarrow$(ii).
\end{proof}

\section{Consequences}\label{sec:consequences}
\begin{corollary}\label{cor:conjectures}
The assertion that every $K_5$-minor-free graph has a normal homogeneous cut semigroup is equivalent to Seymour's assertion that every planar $k$-graph is $k$-edge-colourable for every positive integer $k$.
\end{corollary}
\begin{proof}
Apply Theorem~\ref{thm:main} for all $k\geq1$.
\end{proof}

The converse implication in the cut-normality characterization is already known: a graph with a $K_5$ minor has a nonnormal cut semigroup \cite[discussion after Conjecture~3.7]{SS}\cite[Section~2]{Ohsugi}.

\begin{corollary}\label{cor:holes}
For a simple plane triangulation $G$, every point of $\Sat_k(G)\setminus\Dec_k(G)$ gives a planar $k$-graph that is not $k$-edge-colourable.
Conversely, every planar $k$-graph that is not $k$-edge-colourable gives such a point for a simple planar graph.
\end{corollary}
\begin{proof}
The first statement follows from Proposition~\ref{prop:gadget} on $G^*$.
For the second, use \eqref{eq:switch} and Lemma~\ref{lem:simplify}; a cut decomposition would give the colouring in Proposition~\ref{prop:reverse}.
\end{proof}

\begin{corollary}\label{cor:eight}
Every simple $K_5$-minor-free graph is saturated at heights $1,\ldots,8$.
In particular, any hole in its homogeneous cut semigroup has height at least nine.
\end{corollary}
\begin{proof}
The cases $k=1,2$ of $\mathsf{Sey}_k$ are elementary: a 1-regular graph is a matching, and the odd-cut condition excludes odd cycle components of a 2-regular graph.
For $3\leq k\leq8$, use the established edge-colouring results recorded in Table~\ref{tab:cases}, and then apply Theorem~\ref{thm:main}.
\end{proof}

\begin{table}[ht]
\centering
\caption{Known cases of $\mathsf{Sey}_k$.}\label{tab:cases}
\begin{tabular}{@{}cl@{}}
\toprule
$k$ & Reference\\
\midrule
1, 2 & Elementary\\
3 & Four-colour theorem; \cite[Section~1]{CES}\\
4, 5 & Guenin's results, as recorded in \cite[Section~1]{CES}\\
6 & Dvo\v r\'ak--Kawarabayashi--Kr\'al', \cite[Theorem~2]{DKK}\\
7 & Chudnovsky--Edwards--Kawarabayashi--Seymour, \cite[Theorem~1.2]{CEKS}\\
8 & Chudnovsky--Edwards--Seymour, \cite[Theorem~1.2]{CES}\\
\bottomrule
\end{tabular}
\end{table}

\begin{remark}
For a fixed $K_5$-minor-free graph, the cut polytope is very ample, so its homogeneous semigroup has only finitely many holes \cite[Theorem~2]{LM}.
This gives saturation at all sufficiently large heights, with a bound depending on the graph.
Laso\'n and Micha\l{}ek also proved that a cut semigroup is seminormal if and only if it is normal \cite[Theorem~3]{LM}.
\end{remark}

\begin{remark}
The argument at \eqref{eq:budget} only requires $\sum_i d_{Q_i}(v)<k+2$: a sum of $k$ positive odd integers below $k+2$ must equal $k$.
Equality at $k+2$ allows the degrees $3,1,\ldots,1$.
\end{remark}

\section{Restrictions to cliques}\label{sec:boundary}
In Lemma~\ref{lem:clique}, the common edge coordinates determine the multiplicities of the restricted cuts.
To describe when this holds, let $B_q$ be the matrix whose columns are $(1,\chi_{\delta_{K_q}(S)})$ with one vertex fixed outside $S$.

\begin{proposition}\label{prop:rank}
For $q\geq1$, $\operatorname{rank}_{\mathbb Q}B_q=1+\binom q2$.
In particular, $B_q$ has full column rank if and only if $q\leq3$.
\end{proposition}
\begin{proof}
Represent the remaining vertex bits by signs $z_1,\ldots,z_{q-1}\in\{\pm1\}$, with root sign one.
Edge coordinates are $(1-z_i)/2$ at the root and $(1-z_i z_j)/2$ otherwise.
Thus the row span is generated by the distinct Boolean characters
\[
 1,\quad z_i\ (1\leq i\leq q-1),\quad z_i z_j\ (1\leq i<j\leq q-1).
\]
Their pairwise products have zero sum over the sign cube unless the characters agree, so they are linearly independent.
There are $1+\binom q2$ such characters.
This equals $2^{q-1}$ for $q\leq3$ and is strictly smaller for $q\geq4$.
At $q=4$, the two multisets of normalized bit patterns
\[
 \{0000,0011,0101,0110\},\qquad
 \{0001,0010,0100,0111\}
\]
are distinct, but both have height four and total coordinate two on every edge.
Thus the restricted cut multiplicities need not be determined by their edge sums and total number.
\end{proof}

For larger cliques, the pairing argument of Lemma~\ref{lem:clique} still applies whenever the two decompositions have the same multiplicities for every restricted cut.

\section*{Declaration}
GPT was used in generating proofs.

\medskip
\noindent Numbered citations refer to the preprint versions listed below, except for Edmonds's original article.

{\small
\bibliographystyle{plain}
\bibliography{dcsb}
}
\end{document}